\documentclass[12pt]{article}
\usepackage{mathtools}
\usepackage[english]{babel}
\usepackage{amsthm}
\usepackage{amssymb}
\usepackage{hyperref}
\theoremstyle{definition}
\newtheorem{theorem}{Theorem}[section]
\newtheorem{lemma}[theorem]{Lemma}
\newtheorem{definition}{Definition}[section]
\newtheorem{corollary}{Corollary}[theorem]
\title{Effects From Extra Die Rolls and Choosing the Highest or Lowest}

\author{Fan Jiang, Elvin Jiang}

\begin{document}

\maketitle
\begin{abstract}
This paper looks into the gain or loss from rolling a fair die multiple times and choosing the highest or lowest number as the outcome over rolling the die just once. 
Specifically, this paper gives a general formula for the expected value of choosing the highest or lowest value of any number of die rolls and sides. 
It also shows how, for a fixed number of rolls, the ratio between this expected value and the number of sides converges as the number of sides increases asymptotically.
The converging behavior of this ratio helps formulate the aforementioned gain or loss. 
\end{abstract}
\section*{Introduction} \indent \indent
The expected value of rolling a fair $s$-sided die once is simply $\frac{s+1}{2}$.
However, many dice-rolling games, such as Dungeons and Dragons, include a mechanism that allows the player to roll the die two or more times in order to give an advantage (or disadvantage) to the player, and thus skewing the expected value.
\newline \indent
In a game like Dungeons and Dragons, a player is often prompted to roll a 20-sided die, and have that value compared against some threshold to see if an action goes through. 
An extra die roll in the game is called "advantage" and allows the player to choose the highest value from the two rolls. For example, if a player, who needs a 13 or higher to complete an action, rolls a 7 and a 20 with advantage, the player would succeed in the action as 20 is greater than the 13. 
Conversely, "disadvantage" in the game also has the player roll the die one extra time, but must choose the lower value. 
Were the players to have disadvantage in that same scenario as above, their action would fail as 7 is less than 13.
\newline \indent
This topic was discussed in a YouTube video by Stand-up Maths \cite{JaneStreet}. 
The video claimed that, as the number of sides $s$ approaches $\infty$, the ratio between the expected value of picking the highest value out of fixed $r$ rolls and $s$ is $\frac{r}{r+1}$. 
This paper will expand upon the work in the video by proving a general formula for the expected value of rolling a fair $s$-sided die $r$ times and choosing the highest or lowest as the outcome.
A formal proof will be given for the aforementioned ratio in the YouTube video for all $r$ and $s$.
\newline \indent
Section 1 of this paper sets up the problem and defines the basic concepts that will be used throughout the paper. 
The formula for calculating the frequency of each random outcome in different experiments is presented in Section 2.
In the same section, the expected values of the random experiments will be found. 
In Section 3, asymptotic results for the expected values will be studied when the number of sides approaches infinity.
Finally, conclusions will be made in Section 4.

%Problem Setup
\section{Problem Setup}
\subsection{Definitions}
The die being rolled hereafter is a fair $s$-sided die with each value in the set $\{1, \cdots, s\}, s\geq2$, having an equal probability of being chosen. 
\begin{definition}
The base experiment is a random experiment in which the die is rolled once.
\end{definition}
Let $X$ be the random variable for the base experiment outcomes.
The expected value of $X$ is simply 
\begin{gather*}
E\left[X\right]=\frac{s+1}{2} 
\end{gather*}
\begin{definition}
An advantage experiment is a random experiment in which the die is rolled $r$, $r>1$, times and the largest number is chosen to be the outcome of the experiment.
\newline \indent Let $H_{r,s}$ be the random variable for the advantage experiment outcomes and $E\left[H_{r,s}\right]$ be the expected value of $H_{r,s}$.
\end{definition}
\begin{definition}
A disadvantage experiment is a random experiment in which a fair $s$-sided die is rolled $r$ times and the smallest number is chosen to be the outcome of the experiment.
\newline\indent Let $L_{r,s}$ be the random variable for the disadvantage experiment outcomes and $E\left[L_{r,s}\right]$ be the expected value of $L_{r,s}$.
\end{definition}
\emph{Example 1.} 
A 6-sided die is rolled once and returns a value of 3. 
This is the base experiment.
The die is then rolled twice, returning a 4 and a 1, respectively.
In an advantage experiment, 4 is chosen to be the outcome and a gain of 1 is realized.
In a disadvantage experiment, 1 is chosen to be the outcome and a loss of 2 is realized.

\section{Expected Values of Advantage and Disadvantage Experiments}
%Advantage Frequency
\indent In order to calculate the expected value of either experiment, the frequency of each outcome must be calculated.
The following lemma shows how this frequency can be found.
\begin{lemma}
In an advantage experiment, the frequency of outcome $i,i=1,\cdots,s$, denoted as $\phi_i$, is \\
\begin{gather*}
\phi_i=\sum_{j=1}^{r}{r\choose j}(i-1)^{r-j}
\end{gather*}
and $\phi_i$ can be found as 
\begin{gather*}
\phi_i=i^r-(i-1)^r.
\end{gather*}
\end{lemma}

\begin{proof}
The value of $\phi_i$ can be found by counting.
In the advantage experiment, to get outcome $i$ in $r$ rolls of a die, at least one roll must return $i$, and all the other rolls must return a value less than $i$.
Therefore, the frequency $\phi_i$ can be found using the following summation
\begin{gather*}
\sum_{j=1}^r{r\choose{j}}(i-1)^{r-j}.
\end{gather*}
This states that if $j$, $j=1, \cdots, r$, rolls return the outcome $i$, then the remaining $r-j$ rolls must return a number from 1 to $i-1$.
That is, each of these $r-j$ rolls has $i-1$ options.
\newline \indent To prove the second part of the theorem, note that $i^r-(i-1)^r=((i-1)+1)^r-(i-1)^r$. Using the binomial formula to expand $((i-1)+1)^r$ yields
\begin{gather*}
\begin{split}
((i-1)+1)^r-(i-1)^r
& = \sum_{j=0}^r{r\choose{j}}(i-1)^{r-j}-(i-1)^r\\
& = \sum_{j=1}^r{r\choose{j}}(i-1)^{r-j}\\
& = \phi_i
\end{split}
\end{gather*}
\end{proof}

%Disadvantage frequency
\begin{lemma}
For a disadvantage experiment, the frequency of outcome $i, i=1, \cdots, s$, denoted as $\psi_i$, is
\begin{gather*}
\psi_i=\sum_{j=1}^r{r\choose{j}}(s-i)^{r-j}
\end{gather*}
and $\psi_i$ can be found as
\begin{gather*}
\psi_i=(s-(i-1))^r-(s-i)^r
\end{gather*}
\end{lemma}
\begin{proof}
The value of $\psi_i$ can be found by counting.
In the disadvantage experiment, to get outcome $i$ in $r$ rolls of a die, at least one roll must return $i$, and all the other rolls must return a value greater than $i$.
Therefore, the frequency $\psi_i$ can be found using the following summation.
\begin{gather*}
\sum_{j=1}^r{r\choose{j}}(s-i)^{r-j}
\end{gather*}
This states that if $j$, $j=1, \cdots, r$, rolls return the outcome $i$, then the remaining $r-j$ rolls must return a number from i+1 to $s$.
That is, each of these $r-j$ rolls has $s-i$ options.
\newline \indent To prove the second part of the theorem, note that $(s-(i-1))^r-(s-1)^r=((s-i)+1)^r-(s-i)^r$.
Using the binomial formula to expand $((s-i)+1)^r$ yields,
\begin{gather*}
\begin{split}
((s-i)+1)^r-(s-i)^r & = \sum_{j=0}^r{r\choose{j}}(s-i)^{r-j}-(s-i)^r \\
& = \sum_{j=1}^r{r\choose{j}}(s-i)^{r-j}
\end{split}
\end{gather*}
\end{proof}
%Expected value of AE
After the frequency of each outcome from either an advantage or a disadvantage experiment is found, the expected value of these outcomes can be calculated as shown in the following theorems. 
\begin{theorem}
The expected value of the random variable, $H_{r,s}$, can be found as 
\begin{gather*}
E[H_{r,s}]=s-\frac{\sum_{i=1}^{s-1}i^r}{s^r}.
\end{gather*}
\end{theorem}
\begin{proof}
By the definition of expected value, $E[H_{r,s}]$ can be found as

\begin{gather*}
\begin{split}
E[H_{r,s}]
& = \frac{\sum_{i=1}^si\phi_i}{s^r} \\
& = \sum_{i=1}^si(i^r-(i-1)^r)\\
& = \frac{s^{r+1}-\sum_{i=1}^{s-1}i^r}{s^r} \\
& = s-\frac{\sum_{i=1}^{s-1}i^r}{s^r}
\end{split}
\end{gather*}
where the third equality follows from the telescoping property.
\end{proof}

\begin{corollary}
For any fixed $s$ in an advantage experiment, 
\begin{gather*}
\lim_{r\rightarrow\infty}E[H_{r,s}]= s.
\end{gather*}
\end{corollary}
\begin{proof}
For any fixed s, the greatest value in the summation $\frac{\sum_{i=1}^{s-1}i^r}{s^r}$, is $\frac{(s-1)^r}{s^r}$. 
As $r\rightarrow\infty$, $\frac{(s-1)^r}{s^r}\rightarrow 0$. 
Therefore, all other components of the summation must also approach 0, and the entire summation must approach 0. 
Thus, as $r\rightarrow\infty$, $s-\frac{\sum_{i=1}^{s-1}i^r}{s^r}\rightarrow s$.
\end{proof}

%Expected value of DE
\begin{theorem}
The expected value of the random variable, $L_{r,s}$, can be found as
\begin{gather*}
E[L_{r,s}]=1+\frac{\sum_{i=1}^{s-1}i^r}{s^r}.
\end{gather*}
\end{theorem}
\begin{proof}
By the definition of expected value, $E[L_{r,s}]$ can be found as
\begin{gather*}
\begin{split}
E\left[L_{r,s}\right]
&=\frac{\sum_{i=1}^{s}i\psi(i)}{s^r}\\
&=\frac{\sum_{i=1}^{s}i\left(\left(s-(i-1)\right)^r-(s-i)^r\right)}{s^r}\\
&=\frac{s^r+\sum_{i=1}^{s-1}i^r}{s^r}\\
&=1+\frac{\sum_{i=1}^{s-1}i^r}{s^r}\\
\end{split}
\end{gather*}
\end{proof}
The following corollary follows immediately.
\begin{corollary}
For any fixed $s$ in a disadvantage experiment, as $r\rightarrow\infty$, $E[L_{r,s}]\rightarrow 1$.
\begin{proof}
Following from the proof in Corollary 3.2.1, $1+\frac{\sum_{i=1}^{s-1}i^r}{s^r}$, must approach 1 as $r\rightarrow\infty$
\end{proof}
\end{corollary}
\begin{corollary}
For any $r$ and $s$, $E\left[H_{r,s}\right]+E\left[L_{r,s}\right]=s+1.$ 
%Which shows that the average between the expected value of the advantage experiment and the expected value of the %disadvantage experiment is equivalent to the expected value of the base experiment.
\end{corollary}
\begin{proof}
\begin{equation*}
\begin{split}
E\left[H_{r,s}\right]+E\left[L_{r,s}\right]
&=s-\frac{\sum_{i=1}^{s-1}i^r}{s^r}
+1+\frac{\sum_{i=1}^{s-1}i^r}{s^r}
\\
&=s+1
\end{split}
\end{equation*}
\end{proof}

\section{Expected value as a portion of the number of sides asymptotically}

It is also interesting to see how the expected values behave when $s\rightarrow\infty$ while $r$ is fixed.
Note that, when $s\rightarrow\infty$, the expected values are unbounded.
The following theorems, however, show that the expected values as a portion of the number of sides, i.e., $\frac{E\left[H_{r,s}\right]}{s}$ and $\frac{E\left[L_{r,s}\right]}{s}$, asymptotically converge as $s\rightarrow\infty$.
\begin{theorem}
For an advantage experiment with a fixed number of rolls,
\begin{equation*}
\lim_{s\rightarrow\infty}\frac{E\left[H_{r,s}\right]}{s}=\frac{r}{r+1}.
\end{equation*}
\end{theorem} 
\begin{proof}: Note that, for an advantage experiment,
\begin{gather*}
\begin{split}
\frac{E\left[H_{r,s}\right]}{s}
&=\frac{s-\frac{\sum_{i=1}^{s-1}i^r}{s^r}}{s}\\
&=1-\frac{\sum_{i=1}^{s-1}i^r}{s^{r+1}}\\
&=1-\frac{\frac{1}{r+1}\sum_{j=0}^r(-1)^j{r+1\choose j}B_j(s-1)^{r+1-j}}{s^{r+1}},
\end{split}
\end{gather*}
where the last equality follows from Faulhaber's formula \cite{Faulhaber} and $B_j$ is the $j^{th}$ Bernoulli number \cite{Bernoulli}.
Since $B_0=1$, the equation becomes
\begin{gather*}
\frac{E\left[H_{r,s}\right]}{s}=1-\frac{\frac{1}{r+1}(s-1)^{r+1}-B_1(s-1)^r+\cdots+(-1)^rB_r(s-1)}{s^{r+1}}.
\end{gather*}
In the limit as $s\rightarrow\infty$, 
\begin{gather*}
\begin{split}
\lim_{s\rightarrow\infty}\frac{E\left[H_{r,s}\right]}{s}
& = 1 -\lim_{s\rightarrow\infty}\frac{\frac{1}{r+1}(s-1)^{r+1}-B_1(s-1)^r+\cdots+(-1)^rB_r(s-1)}{s^{r+1}}\\
\end{split}
\end{gather*}
Note that, when $s\rightarrow\infty$, only the highest power of $s$ matters.
Since the coefficient of the $s^{r+1}$ term in the summation is $\frac{1}{r+1}$, the limit simplifies to 
\begin{equation*}
\begin{split}
\lim_{s\rightarrow\infty}\frac{E\left[H_{r,s}\right]}{s}
& = 1-\frac{1}{r+1}\\
& = \frac{r}{r+1}.
\end{split}
\end{equation*}
\end{proof}

\begin{corollary}
For a disadvantage experiment, 
\begin{equation*}
\lim_{s\rightarrow\infty}\frac{E\left[L_{r,s}\right]}{s}=\frac{1}{r+1}.
\end{equation*}
\end{corollary}
\begin{proof}:
By Corollary 2.4.2, $E\left[H_{r,s}\right]+E\left[L_{r,s}\right]=s+1$,
\begin{equation*}
\begin{split}
\lim_{s\rightarrow\infty}\frac{E\left[L_{r,s}\right]}{s}
& = \lim_{s\rightarrow\infty}\frac{s+1-E\left[H_{r,s}\right]}{s}\\
& = \lim_{s\rightarrow\infty}\frac{s+1}{s}-\lim_{s\rightarrow\infty}\frac{E\left[L_{r,s}\right]}{s}\\
& = 1-\frac{r}{r+1}\\
& = \frac{1}{r+1}.
\end{split}
\end{equation*}
\end{proof}

In a game involving die rolling, of particular interest may be how beneficial or harmful these extra die rolls are relative to just one roll.
After defining a relative gain and relative loss, it can be shown that they both converge asymptotically as the number of sides increases.
\begin{definition}
The relative gain in expected value, $\delta_{r,s}$, of a advantage experiment is defined as
\begin{equation*}
\delta_{r,s}=\frac{E\left[H_{r,s}\right]-E[X]}{E[X]}.
\end{equation*} 
\end{definition}
\begin{definition}
The relative loss in expected value, $\lambda_{r,s}$, of a disadvantage experiment is defined as
\begin{equation*}
\lambda_{r,s}=\frac{E[X]-E\left[L_{r,s}\right]}{E[X]}.
\end{equation*} 
\end{definition}
The following corollaries about this relative gain or loss follow readily from the above theorems.

\begin{corollary}
For an advantage experiment, 
\begin{equation*}
\lim_{s\rightarrow\infty}\delta_{r,s}=\frac{r-1}{r+1}.
\end{equation*}
\end{corollary}
\begin{proof}: From the definition of $\delta_{r,s}$,
\begin{equation*}
\begin{split}
\delta_{r,s}
&= \frac{E\left[H_{r,s}\right]-E[X]}{E[X]}\\
&= \frac{E\left[H_{r,s}\right]}{E[X]}-1\\
&=\frac{E\left[H_{r,s}\right]}{\frac{s+1}{2}}-1\\
&=\frac{2E\left[H_{r,s}\right]}{s+1}-1.
\end{split}
\end{equation*} 

Thus, as $s\rightarrow\infty$,
\begin{equation*}
\begin{split}
\lim_{s\rightarrow\infty}\delta_{r,s}
&=2\lim_{s\rightarrow\infty}\frac{E\left[H_{r,s}\right]}{s+1}-1\\
&=2\lim_{s\rightarrow\infty}\frac{E\left[H_{r,s}\right]}{s}-1\\
&=2\frac{r}{r+1}-1\\
&=\frac{r-1}{r+1},
\end{split}
\end{equation*} 
\end{proof}

\begin{corollary}
For a disadvantage experiment, 
\begin{equation*}
\lim_{s\rightarrow\infty}\lambda_{r,s}=\frac{r-1}{r+1}.
\end{equation*}
\end{corollary}
\begin{proof}: From the definition of $\lambda_{r,s}$,
\begin{equation*}
\begin{split}
\lambda_{r,s}
&= \frac{E[X]-E\left[L_{r,s}\right]}{E[X]}\\
&= 1-\frac{E\left[L_{r,s}\right]}{E[X]}\\
&=1-\frac{2E\left[L_{r,s}\right]}{s+1}.
\end{split}
\end{equation*} 

Thus, as $s\rightarrow\infty$,
\begin{equation*}
\begin{split}
\lim_{s\rightarrow\infty}\lambda_{r,s}
&=1-2\lim_{s\rightarrow\infty}\frac{E\left[H_{r,s}\right]}{s+1}\\
&=1-2\lim_{s\rightarrow\infty}\frac{E\left[H_{r,s}\right]}{s}\\
&=1-2\frac{1}{r+1}\\
&=\frac{r-1}{r+1}.
\end{split}
\end{equation*} 

\end{proof}

\section{Conclusions}

\indent In many die-rolling games such as Dungeons and Dragons, the player may roll an $s$-sided die extra times and choose the highest or lowest number among these rolls as the outcome. 
In this paper, the effects of choosing the highest or lowest of these rolls on the expected value of the random outcomes are investigated.
The frequency of each possible outcome out of any number of rolls of a fair die with any number of sides is first found.
Knowing how to calculate this frequency allows the computation of the expected values. 
It is shown that the expected value of an advantage experiment approaches $s$ and that of a disadvantage experiment approaches 1 as the number of rolls increases. 
After the expected values of these random experiments are obtained, it was found that, when the number of rolls is fixed, their values as a portion of the number of sides converges asymptotically with the number of sides.
Namely, the expected value of the outcomes from an advantage experiment as a portion of the number of sides converges to $\frac{r}{r+1}$ and that of a disadvantage experiment goes to $\frac{1}{r+1}$ asymptotically.
As a consequence of the above convergence, it was also shown that the relative gain or loss in expected value of these random experiments over the base experiment converge as well.


\begin{thebibliography}{4}

\bibitem{JaneStreet}
Matt Parker: Stand-up Maths (2022), "The unexpected logic behind rolling multiple dice and picking the highest". YouTube. \url{https://www.youtube.com/watch?v=X_DdGRjtwAo}

\bibitem{Faulhaber}
Donald E. Knuth (1993), "Johann Faulhaber and sums of powers", \emph{Mathematics of Computation}, 61 (203):277-294.

\bibitem{Bernoulli}
Abramowitz, M. and Stegun, I. A. (1972). "Bernoulli and Euler Polynomials and the Euler-Maclaurin Formula", \emph{Handbook of Mathematical Functions with Formulas, Graphs, and Mathematical Tables} (9th printing), New York: Dover Publications, pp. 804-806.

\end{thebibliography}
\end{document}